\documentclass[reqno]{amsart}

\usepackage{amssymb,amsthm,amsmath,tikz,xcolor}
\usepackage[colorlinks=true,citecolor=blue,linkcolor=blue,urlcolor=blue]{hyperref}
\usetikzlibrary{arrows,matrix}

\usepackage[margin=1.25in]{geometry}

\newcommand{\g}{\mathfrak{g}}
\newcommand{\HH}{\mathcal{H}}
\newcommand{\inner}[2]{\left\langle #1, #2 \right\rangle}
\newcommand{\iso}{\cong}
\renewcommand{\mid}{:}

\newcommand{\RC}{\operatorname{RC}} 
\newcommand{\wt}{\mathrm{wt}}
\newcommand{\ZZ}{\mathbf{Z}}
\newcommand{\virtual}[1]{\widehat{#1}}

\newcommand{\lcm}{\operatorname{lcm}} 

\definecolor{darkred}{rgb}{0.7,0,0} 
\newcommand{\defn}[1]{{\color{darkred}\emph{#1}}} 

\usepackage{xparse}

\makeatletter
\protected\def\specialmergetwolists{%
  \begingroup
  \@ifstar{\def\cnta{1}\@specialmergetwolists}
    {\def\cnta{0}\@specialmergetwolists}%
}
\def\@specialmergetwolists#1#2#3#4{%
  \def\tempa##1##2{%
    \edef##2{%
      \ifnum\cnta=\@ne\else\expandafter\@firstoftwo\fi
      \unexpanded\expandafter{##1}%
    }%
  }%
  \tempa{#2}\tempb\tempa{#3}\tempa
  \def\cnta{0}\def#4{}%
  \foreach \x in \tempb{%
    \xdef\cnta{\the\numexpr\cnta+1}%
    \gdef\cntb{0}%
    \foreach \y in \tempa{%
      \xdef\cntb{\the\numexpr\cntb+1}%
      \ifnum\cntb=\cnta\relax
        \xdef#4{#4\ifx#4\empty\else,\fi\x#1\y}%
        \breakforeach
      \fi
    }%
  }%
  \endgroup
}
\makeatother

\usepackage{xparse}
\DeclareDocumentCommand\rpp{ m m g }{
	\foreach \x [count=\s from 1] in {#1}{
	        {\ifnum\s=1
	                \draw (0,-\s)--(\x,-\s);
	                \fi}
	   \draw (0,-\s-1) to (\x,-\s-1);
	   \foreach \y in {0, ..., \x} {\draw (\y,-\s)--(\y,-\s-1);}
	}
	\specialmergetwolists{/}{#1}{#2}\ziplist
	\foreach \x/\y [count=\yi from 1] in \ziplist{
	    \node[anchor=west,font=\scriptsize] at (\x,-\yi - .5) {$\y$};
	}
	\IfValueT {#3}
	{\foreach \z [count=\zi from 1] in {#3} {\node[anchor=east,font=\scriptsize] at (0,-\zi - .5) {$\z$};}}
	{}
}

\theoremstyle{plain}
\newtheorem{thm}{Theorem}[section]
\newtheorem{lemma}[thm]{Lemma}
\newtheorem{conj}[thm]{Conjecture}
\newtheorem{prop}[thm]{Proposition}

\theoremstyle{definition}
\newtheorem{dfn}[thm]{Definition}
\newtheorem{ex}[thm]{Example}
\newtheorem{remark}[thm]{Remark}
\numberwithin{equation}{section}
\numberwithin{figure}{section}
\numberwithin{table}{section}

\title{Rigged configurations for all symmetrizable types}

\author{Ben Salisbury}
\address{Department of Mathematics, Central Michigan University, Mt.\ Pleasant, MI 48859}
\email{ben.salisbury@cmich.edu}
\urladdr{http://people.cst.cmich.edu/salis1bt/}
\thanks{B.S.\ was partially supported by CMU Early Career grant \#C62847 and by Simons Foundation grant \#429950.}

\author{Travis Scrimshaw}
\address{Department of Mathematics, University of Minnesota, Minneapolis, MN 55455}
\email{tscrimsh@umn.edu}
\urladdr{http://math.umn.edu/~tscrimsh/}
\thanks{T.S.\ was partially supported by NSF grant OCI-1147247 and RTG grant NSF/DMS-1148634.}

\begin{document}

\maketitle

\begin{abstract}
In an earlier work, the authors developed a rigged configuration model for the crystal $B(\infty)$ (which also descends to a model for irreducible highest weight crystals via a cutting procedure).  However, the result obtained was only valid in finite types, affine types, and simply-laced indefinite types.  In this paper, we show that the rigged configuration model proposed does indeed hold for all symmetrizable types. As an application, we give an easy combinatorial condition that gives a Littlewood-Richardson rule using rigged configurations which is valid in all symmetrizable Kac-Moody types.

\end{abstract}

\maketitle

\section{Introduction}

The theory of crystal bases~\cite{K91} has provided a natural combinatorial framework to study the representations of Kac-Moody algebras (including classical Lie algebras) and their associated quantum groups.  Their applications span many areas of mathematics, and these diverse applications have compelled researchers to develop different combinatorial models for crystals which yield suitable settings to studying a particular aspect of the representation theory.  See, for example,~\cite{GL05,KN94,KS97,LP08,L95-2}.  The choice of using one model over the other usually depends on the underlying question at hand (and/or on the preference of the author).

We will be using \emph{rigged configurations}, whose origins lie in statistical mechanics. Specifically, they correspond naturally to the eigenvalues and eigenvectors of a Hamiltonian of a statistical model via the Bethe ansatz~\cite{B31,KKR86,KR86}.  As shown in~\cite{SS2015}, the rigged configuration model for $B(\infty)$ has simple combinatorial rules for describing the crystal structure which work in all finite, affine, and all simply-laced Kac-Moody types. These combinatorial rules are only based on the nodes of the Dynkin diagram and their neighbors.

If instead we use a corner transfer matrix approach to solve the Hamiltonian, the eigenvectors become indexed by one-dimensional lattice paths~\cite{B89,HKKOTY99,HKOTT02,NY97,SW99}, which can be interpreted as highest weight vectors in a tensor product of certain crystals known as Kirillov-Reshetikhin crystals. While not mathematically rigorous, these two different approaches suggests a bijection that has been constructed in numerous special cases. See, for example,~\cite{KKR86, KR86, KSS1999, OSS03II, OSSS16, SS15, Scrimshaw15}.
In~\cite{SS2016}, this bijection was extended to show that the rigged configuration model in~\cite{SS2015} and the marginally large tableaux model in~\cite{HL08} agree (for the appropriate types).

The purpose of this paper is to extend the crystal structure on rigged configurations $B(\infty)$ in terms of rigged configurations to all symmetrizable Kac-Moody types.  There are several known models for the crystal $B(\infty)$ in finite and affine types, but only a select few which are uniformly constructed to include all symmetrizable types (e.g., modified Nakajima monomials \cite{KKS07} and Littelmann paths \cite{L95-2}).  Having another model which works beyond finite and affine types is beneficial to studying the combinatorics of the associated representations, which, for example, has come up in the theory of automorphic forms (see \cite{SS2016} for an application of rigged configurations in finite type in this direction).

In~\cite{SS2015}, our proof relied on Schilling's result~\cite{S06} that the crystal structure on rigged configurations satisfied the Stembridge axioms~\cite{Stembridge03}. While the Stembridge axioms are necessary (local) conditions for highest weight crystals, they are only sufficient conditions in simply-laced types.  Then we used the technology of virtual crystals and well-known diagram foldings to extend our results to the other finite and affine types. Since rigged configurations are well-behaved under virtualization~\cite{SS2015, SS15}, the problem of showing rigged configurations model highest weight crystals and $B(\infty)$ for general symmetrizable type is reduced to determining a realization of every symmetrizable type as a diagram folding of a simply-laced type.

It is known that every Cartan type can be realized using a simple graph together with a graph automorphism~\cite{Lusztig93}.  We can realize this graph as a Dynkin diagram of a symmetric type, where the number of edges between vertices $v_i$ and $v_j$ gives the (negative of the) $(i,j)$-entry of the corresponding symmetric Cartan matrix, and the automorphism as a digram folding. Therefore, we can use the corresponding embedding of root lattices and~\cite[Thm.~5.1]{K96} to show there exists a virtualization of a crystal of any symmetrizable type into a crystal of symmetric type.  An explicit virtualization map using Nakajima's and Lusztig's quiver varieties was proven in~\cite{Savage05} in this case.

In this note, we modify the construction in~\cite{Lusztig93} so that the resulting graph is simple, where it can be considered as a simply-laced type. We then use the aforementioned virtualization map to prove an open conjecture (see Conjecture~\ref{conj:RC_virtualization}) stated by the authors that the rigged configuration model for $B(\infty)$ and highest weight crystals $B(\lambda)$ defined in~\cite{SS2015} may be extended to the case of arbitrary symmetrizable Kac-Moody algebras. Furthermore, we expect our results could to lead to a solution to the open problem of determining an analog of the Stembridge axioms for non-simply-laced types, where the only known results are for type $B_2$~\cite{DKK09,Sternberg07}. Indeed, our results allow for a direct link between the crystal operators and a simply-laced type where the Stembridge axioms apply.

The organization of the paper goes as follows.  In Section \ref{sec:background}, we set our notation and recall basic notions about crystals, rigged configurations, and virtualization.  In Section \ref{sec:fold}, we define the diagram folding required to prove our conjecture from \cite{SS2015} in Section \ref{sec:generalRC}.  Lastly, in Section \ref{sec:LR}, we stage the famous Littlewood-Richardson rule for decomposing tensor products of irreducible highest weight crystals in terms of the rigged configuration model.

\section{Background}
\label{sec:background}

We give a background on crystals, virtual crystals, and rigged configurations.

\subsection{Crystals}

Let $\g$ be a symmetrizable Kac-Moody algebra with index set $I$, generalized Cartan matrix $A = (A_{ij})_{i,j\in I}$, weight lattice $P$, root lattice $Q$, fundamental weights $\{\Lambda_i \mid i \in I\}$, simple roots $\{\alpha_i \mid i\in I\}$, and simple coroots $\{h_i \mid i\in I\}$.  There is a canonical pairing $\langle\ ,\ \rangle\colon P^\vee \times P \longrightarrow \ZZ$ defined by $\langle h_i, \alpha_j \rangle = A_{ij}$, where $P^{\vee}$ is the dual weight lattice.

An \defn{abstract $U_q(\g)$-crystal} is a nonempty set $B$ together with maps
\[
\wt \colon B \longrightarrow P, \ \ \
\varepsilon_i, \varphi_i\colon B \longrightarrow \ZZ\sqcup\{-\infty\}, \ \ \
e_i, f_i \colon B \longrightarrow B\sqcup\{0\},
\]
satisfying certain conditions.
The $e_i$ and $f_i$ for $i \in I$ are referred to as the \defn{Kashiwara raising} and \defn{Kashiwara lowering operators}, respectively. See \cite{HK02,K91} for details.  The models used in this paper will be specific, and therefore we will give details related to those models in the subsequent sections.

We say an abstract $U_q(\g)$-crystal is simply a \defn{$U_q(\g)$-crystal} if it is crystal isomorphic to the crystal basis of an integrable $U_q(\g)$-module.

Again let $B_1$ and $B_2$ be abstract $U_q(\g)$-crystals.  The tensor product $B_2 \otimes B_1$ is defined to be the Cartesian product $B_2\times B_1$ equipped with crystal operations defined by
\begin{subequations}
\label{eq:tensor_product_rule}
\begin{align}
\label{eq:tensor_product_rule_e} e_i(b_2 \otimes b_1) &= \begin{cases}
e_i b_2 \otimes b_1 & \text{if } \varepsilon_i(b_2) > \varphi_i(b_1) \\
b_2 \otimes e_i b_1 & \text{if } \varepsilon_i(b_2) \le \varphi_i(b_1),
\end{cases} \\
\label{eq:tensor_product_rule_f} f_i(b_2 \otimes b_1) &= \begin{cases}
f_i b_2 \otimes b_1 & \text{if } \varepsilon_i(b_2) \ge \varphi_i(b_1) \\
b_2 \otimes f_i b_1 & \text{if } \varepsilon_i(b_2) < \varphi_i(b_1),
\end{cases} \\ 
\label{eq:tensor_product_rule_weight} \wt(b_2 \otimes b_1) &= \wt(b_2) + \wt(b_1).
\end{align}
\end{subequations}

\begin{remark}
Our convention for tensor products is opposite the convention given by Kashiwara in~\cite{K91}.
\end{remark}

\subsection{Rigged configurations}

Set $\HH = I \times \ZZ_{>0}$. Consider $\lambda \in P^+ \cup \{\infty\}$ and a sequence of partitions $\nu = (\nu^{(a)} \mid a \in I)$. Let $m_i^{(a)}$ be the number of parts of length $i$ in $\nu^{(a)}$. Define the \defn{vacancy numbers} of $\nu$ to be 
\begin{equation}
\label{eq:vacancy}
p_i^{(a)}(\nu; \lambda) = p_i^{(a)} = c^{(a)} - \sum_{(b,j) \in \HH} A_{ab} \min(i, j) m_j^{(b)},
\end{equation}
where $\lambda = \sum_{a \in I} c^{(a)} \Lambda_a$ or $c^{(a)} = 0$ if $\lambda = \infty$. In addition, we can extend the vacancy numbers to
\[
p_{\infty}^{(a)} = c^{(a)} - \sum_{b \in I} A_{ab} \lvert \nu^{(b)} \rvert
\]
as the limit of $p_i^{(a)}$ as $i \to \infty$ since $\sum_{j=1}^{\infty} \min(i,j) m_j^{(b)} = \lvert \nu^{(b)} \rvert$ for $i \gg 1$.

Recall that a partition is a multiset of integers (typically sorted in weakly decreasing order).  More generally, a \defn{rigged partition} is a multiset of pairs of integers $(i, x)$ such that $i > 0$ (typically sorted under weakly decreasing lexicographic order).  Each $(i,x)$ is called a \defn{string}, while $i$ is called the length or size of the string and $x$ is the \defn{rigging}, \defn{label}, or \defn{quantum number} of the string.  Finally, a \defn{rigged configuration} is a pair $(\nu, J)$ where $J = \big( J_i^{(a)} \big)_{(a, i) \in \HH}$, where each $J_i^{(a)}$ is a weakly decreasing sequence of riggings of strings of length $i$ in $\nu^{(a)}$. We call a rigged configuration \defn{$\lambda$-valid}, for $\lambda \in P^+ \cup \{\infty\}$, if every label $x \in J_i^{(a)}$ satisfies the inequality $p_i^{(a)}(\nu; \lambda) \geq x$ for all $(a, i) \in \HH$. We say a rigged configuration is \defn{highest weight} if $x \geq 0$ for all labels $x$. Define the \defn{colabel} or \defn{coquantum number} of a string $(i,x)$ to be $p_i^{(a)} - x$.  
For brevity, we will often denote the $a$th part of $(\nu,J)$ by $(\nu,J)^{(a)}$ (as opposed to $(\nu^{(a)},J^{(a)})$).

\begin{dfn}
Let $(\nu_\emptyset,J_\emptyset)$ be the rigged configuration with empty partition and empty riggings. Define $\RC(\infty)$ to be the graph generated by $(\nu_\emptyset,J_\emptyset)$, $e_a$, and $f_a$, for $a \in I$, where $e_a$ and $f_a$ acts on elements $(\nu,J)$ in $\RC(\infty)$ as follows. Fix $a \in I$ and let $x$ be the smallest label of $(\nu,J)^{(a)}$.
\begin{itemize}
\item[\defn{$e_a$}:] If $x \geq 0$, then set $e_a(\nu,J) = 0$. Otherwise, let $\ell$ be the minimal length of all strings in $(\nu,J)^{(a)}$ which have label $x$.  The rigged configuration $e_a(\nu,J)$ is obtained by replacing the string $(\ell, x)$ with the string $(\ell-1, x+1)$ and changing all other labels so that all colabels remain fixed.
\item[\defn{$f_a$}:] If $x > 0$, then add the string $(1,-1)$ to $(\nu,J)^{(a)}$.  Otherwise, let $\ell$ be the maximal length of all strings in $(\nu,J)^{(a)}$ which have label $x$.   Replace the string $(\ell, x)$ by the string $(\ell+1, x-1)$ and change all other labels so that all colabels remain fixed.
\end{itemize}
The remaining crystal structure on $\RC(\infty)$ is given by
\begin{subequations}\label{epphiwt}
\begin{align}
\label{Binf_ep} \varepsilon_a(\nu,J) &= \max\{ k \in \ZZ_{\ge0} \mid  e_a^k(\nu,J) \neq 0 \}, \\ 
\label{Binf_phi} \varphi_a(\nu,J) &= \varepsilon_a(\nu,J) + \langle h_a,\wt(\nu,J)\rangle, \\ 
\label{Binf_wt} \wt(\nu,J) &= -\sum_{(a,i)\in\HH} im_i^{(a)}\alpha_a = -\sum_{a\in I} |\nu^{(a)}|\alpha_a.
\end{align}
\end{subequations}
\end{dfn}

It is worth noting that, in this case, the definition of the vacancy numbers reduces to
\begin{equation}
p_i^{(a)}(\nu) = p_i^{(a)} = -\sum_{(b,j) \in \HH} A_{ab} \min(i, j) m_j^{(b)}.
\end{equation}
Moreover, we have $\inner{h_a}{\wt(\nu,J)} = p_{\infty}^{(a)}$ from the crystal structure.

\begin{ex}
Let $\g$ be of type $A_1$, then $(\nu, J) \in \RC(\infty)$ given by $(\nu, J) = f_1^k (\nu_{\emptyset}, J_{\emptyset})$ is the partition $\nu^{(1)} = k$ and the rigging $J_k^{(1)} = \{-k\}$.
\end{ex}

\begin{ex}
The top of the crystal $\RC(\infty)$ in type $A_2$ is shown in Figure~\ref{fig:A2}. We note that we write the rigging on the right of each row and the respective vacancy number on the left.
\end{ex}

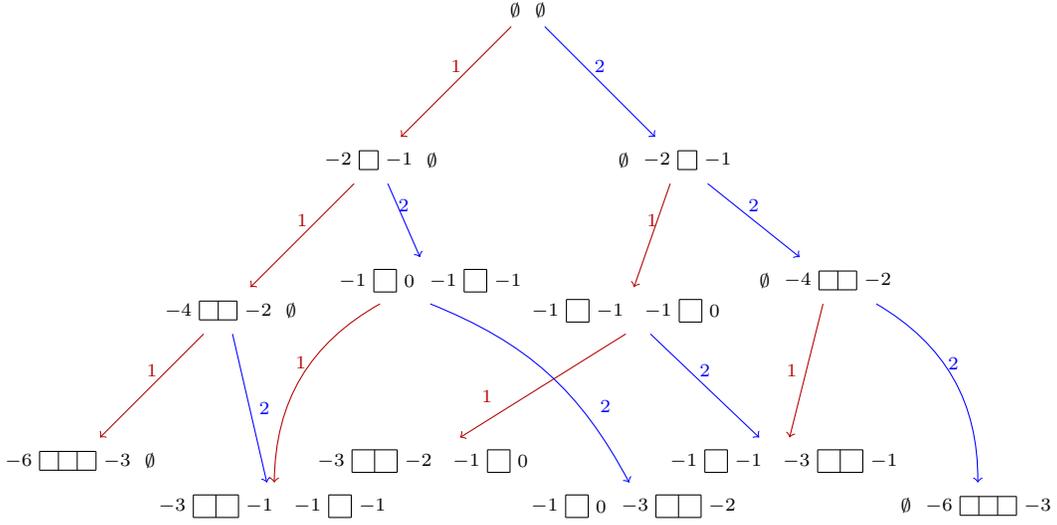
\begin{figure}[ht]
\[
\begin{tikzpicture}[yscale=2,xscale=2,font=\scriptsize]
\node (000) at (0,0) {$\emptyset \ \ \emptyset$};
\node (100) at (-1,-1) 
 {$\begin{tikzpicture}[scale=.25,baseline=-13]
   \rpp{1}{-1}{-2}
  \end{tikzpicture} \ \emptyset$};
\node (010) at (1,-1) 
 {$\emptyset \ \begin{tikzpicture}[scale=.25,baseline=-13]
   \rpp{1}{-1}{-2}
  \end{tikzpicture}$};
\node (200) at (-2,-2) 
 {$\begin{tikzpicture}[scale=.25,baseline=-13]
   \rpp{2}{-2}{-4}
  \end{tikzpicture} \ \emptyset$};
\node (011) at (-.65,-1.8) 
 {$\begin{tikzpicture}[scale=.3,baseline=-16]
   \rpp{1}{0}{-1}
  \begin{scope}[xshift=4cm]
   \rpp{1}{-1}{-1}
  \end{scope}
  \end{tikzpicture}$};
\node (110) at (.65,-2) 
 {$\begin{tikzpicture}[scale=.3,baseline=-16]
   \rpp{1}{-1}{-1}
  \begin{scope}[xshift=5cm]
   \rpp{1}{0}{-1}
  \end{scope}
  \end{tikzpicture}$};
\node (020) at (2,-1.8) 
 {$\emptyset \ 
  \begin{tikzpicture}[scale=.25,baseline=-13]
   \rpp{2}{-2}{-4}
  \end{tikzpicture}$};
\node (300) at (-3,-3) {$\begin{tikzpicture}[scale=.25,baseline=-13]
 \rpp{3}{-3}{-6}
\end{tikzpicture} \ \emptyset$};
\node (111) at (-1.7,-3.3) {$\begin{tikzpicture}[scale=.3,baseline=-16]
 \rpp{2}{-1}{-3}
\begin{scope}[xshift=6cm]
 \rpp{1}{-1}{-1}
\end{scope}
\end{tikzpicture}$};
\node (210) at (-.7,-3) {$\begin{tikzpicture}[scale=.3,baseline=-16]
 \rpp{2}{-2}{-3}
\begin{scope}[xshift=6cm]
 \rpp{1}{0}{-1}
\end{scope}
\end{tikzpicture}$};
\node (021) at (.7,-3.3) {$\begin{tikzpicture}[scale=.3,baseline=-16]
 \rpp{1}{0}{-1}
\begin{scope}[xshift=4cm]
 \rpp{2}{-2}{-3}
\end{scope}
\end{tikzpicture}$};
\node (120) at (1.7,-3) {$\begin{tikzpicture}[scale=.3,baseline=-16]
 \rpp{1}{-1}{-1}
\begin{scope}[xshift=5cm]
 \rpp{2}{-1}{-3}
\end{scope}
\end{tikzpicture}$};
\node (030) at (3,-3.3) {$\emptyset \ \begin{tikzpicture}[scale=.25,baseline=-13]
 \rpp{3}{-3}{-6}
\end{tikzpicture}$};
\path[->,font=\scriptsize,color=darkred]
 (000) edge node[above]{$1$} (100)
 (100) edge node[above]{$1$} (200)
 (010) edge node[above]{$1$} (110)
 (200) edge node[above]{$1$} (300)
 (011) edge[bend right] node[above]{$1$} (111)
 (110.270) edge node[above left,near end]{$1$} (210)
 (020) edge node[left]{$1$} (120);
\path[->,font=\scriptsize,color=blue] 
 (000) edge node[above]{$2$} (010)
 (100) edge node[above]{$2$} (011)
 (010) edge node[above]{$2$} (020)
 (200) edge node[right]{$2$} (111)
 (011.270) edge[bend left=20] node[near end, above right]{$2$} (021.100)
 (110) edge node[above]{$2$} (120)
 (020) edge[bend left] node[above]{$2$} (030);
\end{tikzpicture}
\]
\caption{The top of the crystal $\RC(\infty)$ for $\g = \mathfrak{sl}_3$.}
\label{fig:A2}
\end{figure}

\begin{thm}[\cite{SS2015}]
\label{thm:RCinf_nsl}
Let $\g$ be of simply-laced type. The map defined by $(\nu_\emptyset,J_\emptyset) \mapsto u_\infty$, where $u_\infty$ is the highest weight element of $B(\infty)$, is a $U_q(\g)$-crystal isomorphism $\RC(\infty) \cong B(\infty)$.
\end{thm}

We can extend the crystal structure on rigged configurations to model $B(\lambda)$ as follows. We consider the subcrystal $\RC(\lambda) := \{ (\nu, J) \in \RC(\infty) \mid \text{$(\nu, J)$ is $\lambda$-valid}\}$ for any $\lambda \in P^+$. We have to modify the definition of the weight to be $\wt'(\nu, J) = \wt(\nu, J) + \lambda$. Thus the crystal operators become $f_a(\nu, J) = 0$ if $\varphi_a(\nu, J) = 0$, or equivalently if the result under $f_a$ above is not a $\lambda$-valid rigged configuration. This arises from the natural projection of $B(\infty) \longrightarrow B(\lambda)$.

\subsection{Virtual crystals}

A \defn{diagram folding} is a surjective map $\phi\colon \virtual{I} \longrightarrow I$ between index sets of Kac-Moody algebras and a set $(\gamma_a \in \ZZ_{>0} \mid a \in I)$ of \defn{scaling factors}.  One may induce a map from $\phi$ on the corresponding weight lattices $\Psi \colon P \longrightarrow \virtual{P}$ by asserting
\begin{equation}
\label{eq:weight_embedding}
\Lambda_a \mapsto \gamma_a \sum_{b \in \phi^{-1}(a)} \virtual{\Lambda}_b.
\end{equation}
If $\g \lhook\joinrel\longrightarrow \virtual{\g}$ is the corresponding embedding of symmetrizable Kac-Moody algebras, then it induces an injection $v \colon B(\lambda) \lhook\joinrel\longrightarrow B(\lambda^v)$ as sets, where $\lambda^v := \Psi(\lambda)$. However, there is additional structure on the image under $v$ as a \defn{virtual crystal}, where $e_a$ and $f_a$ are defined on the image as
\begin{equation}
\label{eq:virtual_crystal_ops}
e^v_a = \prod_{b \in \phi^{-1}(a)} \virtual{e}_b^{\,\gamma_a}
\quad \quad
\text{ and }
\quad \quad
f^v_a = \prod_{b \in \phi^{-1}(a)} \virtual{f}_b^{\,\gamma_a},
\end{equation}
respectively, and commute with $v$~\cite{baker2000,OSS03III,OSS03II}. These are known as the \defn{virtual Kashiwara (crystal) operators}. It is shown in~\cite{K96} that for any $a \in I$ and $b,b^{\prime} \in \phi^{-1}(a)$ we have $e_b e_{b^{\prime}} = e_{b^{\prime}} e_b$ and $f_b f_{b^{\prime}} = f_{b^{\prime}} f_b$ as operators (recall that $b$ and $b^{\prime}$ are not connected), so both $e^v_a$ and $f^v_b$ are well-defined. The inclusion map $v$ also satisfies the following commutative diagram.
\begin{equation}
\label{eq:virtual_weight}
\begin{tikzpicture}[xscale=4, yscale=1.5, text height=1.9ex, text depth=0.25ex, baseline=.75cm]
\node (1) at (0,1) {$B(\lambda)$};
\node (2) at (1,1) {$B(\lambda^v)$};
\node (3) at (0,0) {$P$};
\node (4) at (1,0) {$\virtual{P}$};
\path[right hook->,font=\scriptsize]
 (1) edge node[above,inner sep=2]{$v$} (2)
 (3) edge node[below,inner sep=1]{$\Psi$} (4);
\path[->,font=\scriptsize]
 (1) edge node[left] {$\wt$} (3)
 (2) edge node[right]{$\virtual{\wt}$} (4);
\end{tikzpicture}
\end{equation}
In~\cite{baker2000}, it was shown that this defines a $U_q(\g)$-crystal structure on the image of $v$.  More generally, we define a virtual crystal as follows.

\begin{dfn}\label{def:virtual}
Consider any symmetrizable types $\g$ and $\virtual{\g}$ with index sets $I$ and $\virtual{I}$, respectively. Let $\phi \colon \virtual{I} \longrightarrow I$ be a surjection such that $b$ is not connected to $b^{\prime}$ for all $b,b^{\prime} \in \phi^{-1}(a)$ and $a \in I$. Let $\virtual{B}$ be a $U_q(\virtual{\g})$-crystal and $V \subseteq \virtual{B}$.  Let $\gamma = (\gamma_a \in \ZZ_{>0} \mid a \in I)$ be the scaling factors. A \defn{virtual crystal} is the quadruple $(V, \virtual{B}, \phi, \gamma)$ such that $V$ has an abstract $U_q(\g)$-crystal structure defined using the Kashiwara operators $e_a^v$ and $f_a^v$ from \eqref{eq:virtual_crystal_ops} above,
\begin{align*}
\varepsilon_a(x) &:= \frac{\virtual{\varepsilon}_b(x)}{\gamma_a}, & 
\varphi_a(x) &:= \frac{\virtual{\varphi}_b(x)}{\gamma_a}, &\text{ for all } b\in \phi^{-1}(a) \text{ and } x \in V,
\end{align*}
and $\wt := \Psi^{-1} \circ \virtual{\wt}$.
\end{dfn}
We say $B$ \defn{virtualizes} in $\virtual{B}$ if there exists a $U_q(\g)$-crystal isomorphism $v \colon B \longrightarrow V$.  The resulting isomorphism is called the \defn{virtualization map}. We denote the quadruple $(V,\virtual{B},\phi,\gamma)$ simply by $V$ when there is no risk of confusion.

The virtualization map $v$ from rigged configurations of type $\g$ to rigged configurations of type $\virtual{\g}$ is defined by
\begin{equation}
\label{eq:virtual_rc}
\virtual{m}_{\gamma_a i}^{(b)} = m_i^{(a)}, \qquad \virtual{J}_{\gamma_a i}^{(b)} = \gamma_a J_i^{(a)}, 
\end{equation}
for all $b \in \phi^{-1}(a)$.  A $U_q(\g)$-crystal structure on rigged configurations is defined by using virtual crystals \cite{OSS03II}. Moreover, we use Equation~\eqref{eq:virtual_rc} to describe the virtual image of the type $\g$ rigged configurations into type $\virtual{\g}$ rigged configurations. Explicitly $(\virtual{\nu}, \virtual{J}) \in V$ if and only if
\begin{enumerate}
\item $\virtual{m}_i^{(b)} = \virtual{m}_i^{(b^{\prime})}$ and $\virtual{J}_i^{(b)} = \virtual{J}_i^{(b^{\prime})}$ for all $b, b^{\prime} \in \phi^{-1}(a)$,
\item $\virtual{J}^{(b)}_i \in \gamma_a \ZZ$ for all $b \in \phi^{-1}(a)$, and
\item $\virtual{m}^{(b)}_i = 0$ and $\virtual{J}^{(b)}_i = 0$ for all $i \notin \gamma_a \ZZ$ for all $b \in \phi^{-1}(a)$.
\end{enumerate}

Next, we recall~\cite[Conj.~5.12]{SS2015}.

\begin{conj}
\label{conj:RC_virtualization}
For all $\g$ of symmetrizable type, there exists a simply-laced type $\virtual{\g}$ and diagram folding $\phi$ with scaling factors $(\gamma_a \in \ZZ_{>0} \mid a \in I)$ such that $\RC(\lambda)$ virtualizes in $\RC(\lambda^v)$ under the virtualization map given by Equation~\eqref{eq:virtual_rc}.
\end{conj}

\section{Symmetrizable types as foldings from simply-laced types}
\label{sec:fold}

In this section, we give a modified graph construction from~\cite[Prop.~14.1.2]{Lusztig93} which ensures that the resulting graph is simple.  We identify simple graphs with simply-laced Dynkin diagrams.

Let $D = (d_a)_{a \in I}$ be a diagonal matrix such that $DA$ is symmetric with $d_a \in \ZZ_{>0}$ and $\gcd(d_a \mid a \in I) = 1$. Define $d_{a,b} := \lcm(d_a, d_b)$ and $N = \max \left\{ \frac{-A_{ab}d_a}{d_{a,b}} \mid a \neq b \in I \right\}$.  Assert that $\Gamma_A$ is a graph with vertex set 
\[
\{v_{a,s} : a \in I \text{ and } s \in \ZZ / (N d_a) \ZZ \}
\]
and edge set constructed as follows.   Fix some $a \neq b \in I$, let $\widetilde{d}_a, \widetilde{d}_b$ be such that $\gcd(\widetilde{d}_a, \widetilde{d}_b) = 1$ and $\frac{\widetilde{d}_a}{\widetilde{d}_b} = \frac{d_a}{d_b}$, and let $c_{ab}$ be such that $c_{ab} \widetilde{d}_a = -A_{ba}$.  Then $\Gamma_A$ has edges
\[
\bigl\{ \{v_{a,s}, v_{b,s+k}\} \mid a,b \in I, \ k = 0, 1, \dotsc, c_{ab}-1, \ s = 0, 1, \dotsc, N {d}_a {d}_b-1 \bigr\},
\]
where the indices $s$ and $s+k$ are taken modulo $N d_a$ and $N d_b$, respectively. Define a map $\phi_A \colon \Gamma_A \longrightarrow \Gamma_A$ by $\phi_A(v_{a,s}) = v_{a,s+1}$ for $a \in I$ and $s+1$ understood modulo $Nd_a$.  

\begin{ex}
Let $A = (\begin{smallmatrix}2&-6\\-4&2\end{smallmatrix})$.  Then $D = (\begin{smallmatrix}2&0\\0&3\end{smallmatrix})$ is a diagonal matrix such that $DA = (\begin{smallmatrix}4&-12\\-12&6\end{smallmatrix})$ is symmetric.  Since $d_1 = 2$ and $d_2 = 3$ are relatively prime, set $\widetilde d_1 = 2$, $\widetilde d_2 = 3$.  Then $N = 2$ and $c_{12} = c_{21} = 2$.  Hence $\Gamma_A$ has vertices 
\[
\{v_{1,0}, v_{1,1}, v_{1,2}, v_{1,3}, v_{2,0}, v_{2,1}, v_{2,2}, v_{2,3}, v_{2,4}, v_{2,5}\}
\] 
and
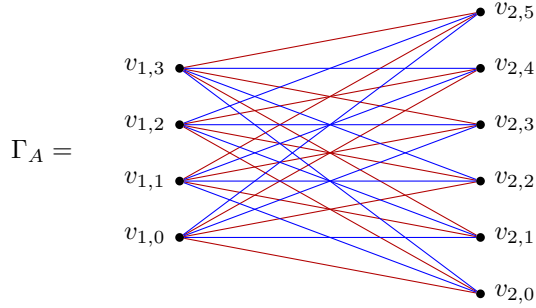
\[
\Gamma_A = \ \ \ \
\begin{tikzpicture}[baseline=30,xscale=2,yscale=.75]
\node[circle,fill,scale=.35,label={west:$v_{1,0}$}] (10) at (0,0) {};
\node[circle,fill,scale=.35,label={west:$v_{1,1}$}] (11) at (0,1) {};
\node[circle,fill,scale=.35,label={west:$v_{1,2}$}] (12) at (0,2) {};
\node[circle,fill,scale=.35,label={west:$v_{1,3}$}] (13) at (0,3) {};
\node[circle,fill,scale=.35,label={east:$v_{2,0}$}] (20) at (2,-1) {};
\node[circle,fill,scale=.35,label={east:$v_{2,1}$}] (21) at (2,0) {};
\node[circle,fill,scale=.35,label={east:$v_{2,2}$}] (22) at (2,1) {};
\node[circle,fill,scale=.35,label={east:$v_{2,3}$}] (23) at (2,2) {};
\node[circle,fill,scale=.35,label={east:$v_{2,4}$}] (24) at (2,3) {};
\node[circle,fill,scale=.35,label={east:$v_{2,5}$}] (25) at (2,4) {};
\path[-, color=darkred]
 (10) edge (20)
 (11) edge (21)
 (12) edge (22)
 (13) edge (23)
 (10) edge (24)
 (11) edge (25)
 (12) edge (20)
 (13) edge (21)
 (10) edge (22)
 (11) edge (23)
 (12) edge (24)
 (13) edge (25);
\path[-, color=blue]
 (10) edge (21)
 (11) edge (22)
 (12) edge (23)
 (13) edge (24)
 (10) edge (25)
 (11) edge (20)
 (12) edge (21)
 (13) edge (22)
 (10) edge (23)
 (11) edge (24)
 (12) edge (25)
 (13) edge (20);
\end{tikzpicture}
\]
\end{ex}

\begin{prop}
\label{prop:graph_properties}
Let $A$ be any symmetrizable Cartan matrix and $a \neq b\in I$.  The map $\phi_A$ defined above is a Dynkin diagram automorphism of $\Gamma_A$.  Moreover, let $E_{ab}$ denote the number of edges between any fixed vertex in the $\phi_A$-orbit of $a$ with some vertex in the $\phi_A$-orbit of $b$. Then
$
-A_{ab} = E_{ab}.
$
\end{prop}

In order to prove this proposition, we require a result from \cite{Cetal:10}, which we restate here for the reader's convenience.

\begin{prop}[{\cite[Prop.~2.5]{Cetal:10}}]
\label{prop:Cetal10}
Let $A$ be a symmetrizable Cartan matrix and let $\Gamma_A$ be its corresponding Dynkin diagram.  Let $D = (d_a)_{a\in I}$ be the (diagonal) symmetrizing matrix of $A$.  Then
\[
\frac{d_a}{d_b} = \frac{A_{ba}}{A_{ab}},
\]
whenever $a$ and $b$ are connected by an edge in $\Gamma_A$.
\end{prop}

We also need the following technical lemma.

\begin{lemma}
\label{lemma:c_bound}
With the notation as above, we have $c_{ab} \leq N$.
\end{lemma}

\begin{proof}
We have
\[
c_{ab} = \frac{-A_{ba}}{\widetilde{d}_a} = \frac{-A_{ab}}{\widetilde{d}_a} \cdot \frac{d_a}{d_b} = \frac{-A_{ab}}{\widetilde{d}_a} \cdot \frac{\widetilde{d}_a}{\widetilde{d}_b} = \frac{-A_{ab}}{\widetilde{d}_b}
\]
from Proposition~\ref{prop:Cetal10}. We also have
\begin{equation}
\label{eq:redefining_N}
\frac{-A_{ab} d_a}{\lcm(d_a, d_b)} = \frac{-A_{ab} \gcd(d_a, d_b)}{d_b} = \frac{-A_{ab}}{\widetilde{d}_b}
\end{equation}
from the definition of $\widetilde{d}_b = \frac{d_b}{\gcd(d_a,d_b)}$. Since $N$ is defined as the maximum over the values given by Equation~\eqref{eq:redefining_N}, the claim follows.
\end{proof}

\begin{proof}[Proof of Proposition~\ref{prop:graph_properties}]
The fact that $\phi_A$ is an automorphism is clear from the construction of $\Gamma_A$ and $\phi_A$. 
For $a\neq b$ in $I$, we have
\[
E_{ab} 
= \#\bigl\{ \{v_{a,s}, v_{b,s+k} \} : k = 0 ,1,\dots, c_{ab}-1, \ s = 0,1,\dots, N d_a d_b - 1 \bmod{N d_a} \bigr\}.
\]
Next, we consider the number of times the edge $\{v_{a,0}, v_{b,0} \}$ occurs in the set above. Note that for $k = 0$, we have $d_{a,b}$ values of $s$ such that $s \equiv 0 \bmod{N d_a}$ and $s \equiv 0 \bmod{N d_b}$. Lemma~\ref{lemma:c_bound} states that $k < N$, and hence, there does not exist a value $k>0$ such that $s \equiv 0 \bmod{N d_a}$ such that $s + k \equiv 0 \bmod{N d_b}$. From the construction, we can take any fixed edge and obtain the same result.
Hence,
\[
E_{ab} 
=  c_{ab} \cdot \frac{N {d}_a {d}_b}{N {d}_a}\cdot \frac{1}{d_{a,b}}
= c_{ab} \cdot {d}_b \cdot \frac{\widetilde{d}_b}{d_b}
= -\frac{A_{ba}}{\widetilde{d}_a} \cdot {d}_b \cdot \frac{\widetilde{d}_b}{d_b}
= -\frac{A_{ab} \widetilde{d}_a}{\widetilde{d}_{a}\widetilde{d}_b} \cdot {d}_b \cdot \frac{\widetilde{d}_b}{d_b}
= -A_{ab},
\]
where we used Proposition \ref{prop:Cetal10} and the fact that $\frac{d_a}{d_b} = \frac{\widetilde{d}_a}{\widetilde{d}_b}$.
\end{proof}

In other words, Proposition~\ref{prop:graph_properties} states that we can recover $A$ from $(\Gamma_A, \phi_A)$. We also have that the induced map on the weight lattice from Equation~\eqref{eq:weight_embedding} implies
\begin{equation}
\label{eq:root_embedding}
\alpha_a \mapsto \sum_{b \in \phi^{-1}_A(a)} \virtual{\alpha}_b.
\end{equation}

\section{Proof of Conjecture \ref{conj:RC_virtualization}}
\label{sec:generalRC}

In this section we prove our main result.  That is, we show that $\RC(\lambda)$ virtualizes in $\RC\bigl(\Psi(\lambda)\bigr)$ for any $\lambda \in P^+ \cup \{\infty\}$, where $\Psi \colon P \longrightarrow \virtual{P}$ is the induced map on the weight lattices corresponding to $\Gamma_A$ with $\gamma_a = 1$ for all $a \in I$.  Indeed, by the definition of the crystal operators, we can restrict the proof to the rank two case.  The fact that the crystal operators commute with the virtualization map can be made using an argument similar to~\cite[Prop.~3.7]{SS15} using Proposition~\ref{prop:graph_properties} and the construction of $\Gamma_A$.

We sketch the argument here. Note that $\virtual{m}_i^{(b)} = \virtual{m}_i^{(b')}$ and $\virtual{J}^{(b)}_i = \virtual{J}_i^{(b')}$ for all $b,b' \in \phi^{-1}(a)$, $a \in I$, and $i \in \ZZ_{>0}$. Hence $e_a^v$ and $f_a^v$ change $\nu^{(b)}$ for all $b \in \phi^{-1}(a)$ in exactly the same position. Moreover, each $\virtual{\nu}^{(b')}$ for $b' \in \phi^{-1}(a')$ has exactly $-A_{aa'}$ values of $b \in \phi^{-1}(a)$ such that $\virtual{A}_{bb'} = -1$ (i.e., $b$ and $b'$ are adjacent in $\Gamma_A$), so when there is a change in vacancy numbers, and hence a change in the riggings, it is exactly $A_{aa'}$ for all $a' \in I$. So $f_a^v(\virtual{\nu}, \virtual{J}) = v\bigl(f_a(\nu, J)\bigr)$.

Thus the result follows from~\cite[Prop.~4.2]{SS2015} (which relies on~\cite[Lemma~3.6]{S06}) and Equation~\eqref{eq:virtual_rc}.
Furthermore, the statements of Theorem~5.20 and Corollary~6.2 in~\cite{SS2015} hold. Hence, we have shown that the rigged configuration model is valid in all symmetrizable types. Alternatively this follows from~\cite[Thm.~5.1]{K96} by Equation~\eqref{eq:root_embedding}.

\begin{thm}
\label{thm:main_result}
Let $\g$ be a Kac-Moody algebra of arbitrary symmetrizable type. Then $\RC(\lambda) \iso B(\lambda)$ for $\lambda \in P^+ \cup \{\infty\}$.
\end{thm}

\section{Littlewood-Richardson rule}
\label{sec:LR}

In this section, we give a Littlewood-Richardson rule using rigged configurations, which requires a combinatorial description of $\varepsilon_a$ and $\varphi_a$.  The proof of the following Proposition follows~\cite{Sakamoto14,SS2015}.

\begin{prop}
\label{prop:ep_phi}
Let $x$ be the smallest rigging in $(\nu, J)^{(a)}$, where $(\nu,J) \in \RC(\lambda)$ or $\RC(\infty)$. Then, for all $a\in I$, 
\[
\varepsilon_a(\nu,J) = -\min\{0, x\},
\ \ \ \ \ \ 
\varphi_a(\nu,J) = p_{\infty}^{(a)} - \min\{0, x\}.
\]
\end{prop}

\begin{thm}
Let $\lambda,\mu \in P^+$ be such that $\lambda = \sum_{a \in I} c_a \Lambda_a$. Then
\[
\RC(\mu) \otimes \RC(\lambda) \iso \bigoplus_{\substack{(\nu, J) \in \RC(\mu) \\ \min\{ \min J_i^{(a)} \mid i \in \ZZ_{>0}\} \ge -c_a\\ \mathrm{for\,all\,} a\in I}} \RC\bigl( \lambda + \wt(\nu, J) \bigr).
\]
\end{thm}

\begin{proof}
Recall that, if $B$ is a crystal, then $v \in B$ is called highest weight if $e_av = 0$ for all $a\in I$.  By \cite[\S4.5]{K95} (or directly from \eqref{eq:tensor_product_rule}), the highest weight elements of $\RC(\mu)\otimes \RC(\lambda)$ are precisely those elements of the form $(\nu,J) \otimes (\nu_\lambda,J_\lambda)$, where $(\nu_\lambda,J_\lambda)$ is the highest weight rigged configuration in $\RC(\lambda)$ and $\varepsilon_a(\nu,J) \le \langle h_a, \lambda\rangle$ for all $a\in I$.  Since $\lambda = \sum_{a\in I} c_a\Lambda_a$, we seek those $(\nu,J) \in \RC(\mu)$ such that $\varepsilon_a(\nu,J) \le c_a$ for all $a\in I$.  By Proposition \ref{prop:ep_phi}, we have $\varepsilon_a(\nu,J) = -\min\{0,x\}$, where $x$ is the smallest rigging in $(\nu,J)^{(a)}$.  The smallest rigging in $(\nu,J)^{(a)}$ is 
\[
\min\bigl\{ \min\{\ell : \ell \in J_i^{(a)}\} : i \in \ZZ_{>0} \bigr\},
\] 
so we require
\begin{align*}
c_a &\ge \varepsilon_a(\nu,J) \\
&= -\min\Bigl\{ 0, \min\bigl\{ \min\{\ell : \ell \in J_i^{(a)}\} : i \in \ZZ_{>0} \bigr\} \Bigr\} \\
&\ge -\min\bigl\{ \min\{\ell : \ell \in J_i^{(a)}\} : i \in \ZZ_{>0} \bigr\},
\end{align*}
which is what we set out to prove.
\end{proof}

\begin{ex}
Suppose $\g$ is of type $A_2$ and let $\lambda = \Lambda_1+\Lambda_2$ and $\mu = \Lambda_1$.  Since $B(\mu)$ is the crystal
\[
\varnothing
\ \ 
\varnothing\ \ 
\xrightarrow{\ \ \ 1 \ \ \ }\ \ 
\begin{tikzpicture}[scale=.35,anchor=top,baseline=-17.5]
 \rpp{1}{-1}{-1}
\end{tikzpicture}\ \ 
\varnothing\ \
\xrightarrow{\ \ \ 2 \ \ \ }\ \ 
\begin{tikzpicture}[scale=.35,anchor=top,baseline=-17.5]
 \rpp{1}{0}{0}
 \begin{scope}[xshift=4cm]
 \rpp{1}{-1}{-1}
 \end{scope}
\end{tikzpicture},
\]
it follows that
\[
B(\mu)\otimes B(\lambda) \cong B(2\Lambda_1+\Lambda_2) \oplus B(2\Lambda_2) \oplus B(\Lambda_1).
\]
\end{ex}

Recall that rigged configurations in finite type can be considered as classical components of $U_q'(\g)$-crystals, where $\g$ is of affine type, isomorphic to $\bigotimes_{i=1}^N B^{r_i,1}$. We note that there is an algorithm to construct all classically highest weight $U_q'(\g)$-rigged configurations given by Kleber in simply-laced types~\cite{kleber98} and extended to all other types by using virtualization~\cite{OSS03III}. It would be interesting to determine which nodes of the Kleber tree correspond to the highest weight elements in $B(\lambda) \otimes B(\mu)$ and more generally $B(\lambda_1) \otimes \cdots \otimes B(\lambda_{\ell})$.

\section*{Acknowledgements}
The authors would like to thank the anonymous referee for very helpful comments, which improved both the quality and clarity of this manuscript.

\bibliography{RC_LR}{}
\bibliographystyle{amsalpha}

\end{document}